\documentclass[11pt,letterpaper,reqno,final]{amsart}

\usepackage[english]{babel}

\usepackage{amsfonts,amsmath,amssymb}
\usepackage{amsthm}
\usepackage{url}
\usepackage{ifthen}
\usepackage{ifpdf}
\usepackage{colonequals}
\usepackage{enumerate}

\ifpdf
\usepackage{aeguill}
\usepackage[pdftex]{graphicx}
\else
\usepackage[dvips]{graphicx}
\fi

\title[Optimization and Unit-Distance Representations of Graphs]{%
  Optimization Problems over\\Unit-Distance Representations of Graphs
}
\author{Marcel K.\ de Carli Silva \and Levent Tun\c{c}el}
\date{\today}

\newtheorem{theorem}{Theorem}[section]
\newtheorem{proposition}[theorem]{Proposition}
\newtheorem{corollary}[theorem]{Corollary}

\ifpdf
\usepackage[pdftex,%
pdfauthor={Marcel K. de Carli Silva and Levent Tuncel},%
pdftitle={Optimization Problems over Unit-Distance Representations of Graphs},%
hypertexnames=true,bookmarks=true,pagebackref]{hyperref}
\else
\fi

\numberwithin{equation}{section}

\DeclareMathOperator{\Diag}{Diag}
\DeclareMathOperator{\Image}{Im}
\DeclareMathOperator{\Null}{Null}
\DeclareMathOperator{\SOCop}{SOC}
\DeclareMathOperator{\conv}{conv}
\DeclareMathOperator{\diag}{diag}
\DeclareMathOperator{\linspan}{span}
\DeclareMathOperator{\thetabody}{TH}
\DeclareMathOperator{\trace}{Tr}
\newcommand{\Acal}{\mathcal{A}}
\newcommand{\Ecal}{\mathcal{E}}
\newcommand{\Gallai}{Gallai}
\newcommand{\Integers}{\mathbb{Z}}
\newcommand{\Laplacian}[2][]{\Lcal_{\ifthenelse{\equal{#1}{}}{G}{#1}}(#2)}
\newcommand{\Lcal}{\mathcal{L}}
\newcommand{\Lovasz}{Lov\'asz}
\newcommand{\Orthraw}{\mathbb{O}}
\newcommand{\Orth}[1][]{\Orthraw^{\ifthenelse{\equal{#1}{}}{n}{#1}}}
\newcommand{\Pd}[1][]{\Symraw_{++}^{\ifthenelse{\equal{#1}{}}{n}{#1}}}
\newcommand{\Psd}[1][]{\Symraw_+^{\ifthenelse{\equal{#1}{}}{n}{#1}}}
\newcommand{\Reals}{\mathbb{R}}
\newcommand{\Saxe}{Saxe}
\newcommand{\Symraw}{\mathbb{S}}
\newcommand{\Sym}[1][]{\Symraw^{\ifthenelse{\equal{#1}{}}{n}{#1}}}
\newcommand{\Szegedy}{Szegedy}
\newcommand{\Ucal}{\mathcal{U}}
\newcommand{\Xb}{\bar{X}}
\newcommand{\Xh}{\hat{X}}
\newcommand{\card}[2][]{#1| #2 #1|}
\newcommand{\drop}{\setminus}
\newcommand{\eps}{\varepsilon}
\newcommand{\ffrom}{\colon}
\newcommand{\fto}{\to}
\newcommand{\iprodt}[2]{\transp{#1}#2}
\newcommand{\iprod}[2]{\langle #1, #2 \rangle}
\newcommand{\myemptyset}{\varnothing}
\newcommand{\myhalf}{\textstyle \frac{1}{2}}
\newcommand{\nbd}{\nobreakdash}
\newcommand{\norm}[2][]{\|#2\|_{#1}}
\newcommand{\ones}{\bar{e}}
\newcommand{\oprodsym}[1]{\oprod{#1}{#1}}
\newcommand{\oprod}[2]{#1#2^{T}}
\newcommand{\paren}[2][]{#1({#2}#1)}
\newcommand{\qform}[2]{\transp{#2}#1#2}
\newcommand{\setst}[3][]{#1\{\,{#2}\,\colon{#3} #1\}}
\newcommand{\set}[2][]{#1\{ {#2} #1\}}
\newcommand{\soc}[1]{\SOCop_{#1}}
\newcommand{\sqbrac}[2][]{#1[{#2}#1]}
\newcommand{\textdefn}[1]{\emph{#1} \index{#1}}
\newcommand{\textdef}[1]{\textdefn{#1}}
\newcommand{\transp}[1]{#1^T}
\newcommand{\yb}{\bar{y}}
\newcommand{\yh}{\hat{y}}
\newcommand{\zb}{\bar{z}}
\newcommand{\zh}{\hat{z}}

\newenvironment{optproblem}[1][]%
  {%
    \begin{equation}
      \ifthenelse{\equal{#1}{}}{}{\label{#1}}
      \begin{array}{rll}
  }%
  {%
      \end{array}
    \end{equation}
    \ignorespacesafterend
  }

\begin{document}

\begin{abstract}
  We study the relationship between unit-distance representations and
  \Lovasz{} theta number of graphs, originally established by
  \Lovasz. We derive and prove min-max theorems. This framework
  allows us to derive a weighted version of the hypersphere number of
  a graph and a related min-max theorem. Then, we connect to sandwich
  theorems via graph homomorphisms. We present and study a
  generalization of the hypersphere number of a graph and the related
  optimization problems. The generalized problem involves finding the
  smallest ellipsoid of a given shape which contains a unit-distance
  representation of the graph. We prove that arbitrary positive
  semidefinite forms describing the ellipsoids yield NP-hard problems.
\end{abstract}

\maketitle

\section{Introduction}
Geometric representation of graphs is a beautiful area where
combinatorial optimization, graph theory and semidefinite optimization
meet and connect with many other research areas. In this paper, we
start by studying geometric representations of graphs where each node
is mapped to a point on a hypersphere so that each edge has unit
length and the radius of the hypersphere is minimum.
\Lovasz~\cite{LovaszSDP} proved that this graph invariant is related
to the \Lovasz{} theta number of the complement of the graph via a
simple but nonlinear equation. We show that this tight relationship
leads to min-max theorems and to a ``dictionary'' to translate
existing results about the theta function and its variants to the
hypersphere representation setting and vice versa.

Based on our approach, we derive a weighted version of the hypersphere
number of a graph and deduce related min-max theorems. Our viewpoint
allows us to make new connections, strengthen some facts and correct
some inaccuracies in the literature.

After observing that the hypersphere number of a graph is equal to the
radius of the smallest Euclidean ball containing a unit-distance
representation of the graph, we propose generalizations of the
underlying optimization problems. Given a graph, the generalized
optimization problem seeks the smallest ellipsoid of given shape which
contains a unit-distance representation of the graph. We finally show
that at this end of the new spectrum of unit-distance representations,
arbitrary positive semidefinite forms describing the shapes of the
ellipsoids yield NP-hard geometric representation problems.

\section{Preliminaries}
\label{sec:prelim}

We denote the set of symmetric $n \times n$ matrices by~$\Sym$, the
set of symmetric $n \times n$ positive semidefinite matrices
by~$\Psd$, and the set of symmetric $n \times n$ positive definite
matrices by~$\Pd$. For a finite set~$V$, the set of symmetric $V
\times V$ matrices is denoted by~$\Sym[V]$, and the symbols~$\Psd[V]$
and~$\Pd[V]$ are defined analogously. For $A,B \in \Sym$, we write $A
\succeq B$ meaning $(A-B) \in \Psd$. Define an inner product on~$\Sym$
by $\iprod{A}{B} \colonequals \trace(AB)$, where $\trace(X)
\colonequals \sum_{i=1}^n X_{ii}$ is the trace of $X \in
\Reals^{n\times n}$. The linear map $\diag \ffrom \Sym \fto \Reals^n$
extracts the diagonal of a matrix; its adjoint is denoted by $\Diag$.

The vector of all ones is denoted by~$\ones$. We abbreviate $[n]
\colonequals \set{1,\dotsc,n}$. The notation $\norm{\cdot}$ for a norm
is the Euclidean norm unless otherwise specified. For a finite set
$V$, the set of orthogonal $V \times V$ matrices is denoted
by~$\Orth[V]$. The set of nonnegative reals is denoted by~$\Reals_+$.
The set of positive reals is denoted by~$\Reals_{++}$. Define the
notations~$\Integers_+$ and~$\Integers_{++}$ analogously for integer
numbers.

For any function $f$ on graphs, we denote by $\overline{f}$ the
function defined by $\overline{f}(G) \colonequals f(\overline{G})$ for
every graph~$G$, where $\overline{G}$ denotes the complement of~$G$.
For a graph~$G$, we denote the clique number of~$G$ by~$\omega(G)$ and
the chromatic number of~$G$ by~$\chi(G)$. The complete graph on~$[n]$
is denoted by~$K_n$.

Let $G$ be a graph. Its vertex set is $V(G)$ and its edge set
is~$E(G)$. For $S \subseteq V(G)$, the subgraph of~$G$ induced by~$S$,
denoted by $G[S]$, is the subgraph of~$G$ on~$S$ whose edges are the
edges of~$G$ that have both ends in~$S$. For $i \in V(G)$, the
neighbourhood of~$i$, denoted by~$N(i)$, is the set of nodes of~$G$
adjacent to~$i$. A block of~$G$ is an inclusionwise maximal induced
subgraph of~$G$ with no cut-nodes, where a cut-node of a graph~$H$ is
a node $i \in V(H)$ such that $H[V(H) \drop \set{i}]$ has more
connected components than~$H$.

For a graph $G = (V,E)$, the \textdef{Laplacian of~$G$} is the linear
extension $\Lcal_G \ffrom \Reals^E \fto \Sym[V]$ of the map
$e_{\set{i,j}} \mapsto \oprodsym{(e_i-e_j)}$ for every $\set{i,j} \in
E$, where $e_i$ denotes the $i$th unit vector. Laplacians arise
naturally in spectral graph theory and spectral geometry; see
\cite{Chung97a}.

\section{Hypersphere representations and the \Lovasz{} theta function}
\label{sec:th}

Let $G = (V,E)$ be a graph. A \textdef{unit-distance representation
  of~$G$} is a function $u \ffrom V \fto \Reals^d$ for some $d \geq 1$
such that $\norm{u(i)-u(j)}=1$ whenever $\set{i,j} \in E$. A
\textdef{hypersphere representation of~$G$} is a unit-distance
representation of~$G$ that is contained in a hypersphere centered at
the origin, and the \textdef{hypersphere number of~$G$,} denoted
by~$t(G)$, is the square of the smallest radius of a hypersphere that
contains a unit-distance representation of~$G$. The \textdef{theta
  number of~$G$} is defined by
\begin{equation}
  \label{eq:theta-sdp}
  \vartheta(G)
  \colonequals
  \max\setst[\big]{
    \qform{X}{\ones}
  }{
    \trace(X) = 1,\,
    X_{ij} = 0\,\forall \set{i,j} \in E,\,
    X \in \Psd[V]
  }.
\end{equation}
This parameter was introduced by \Lovasz{} in the seminal
paper~\cite{Lovasz79a}; see also~\cite{GroetschelLS93a, Knuth94a} for
further properties and alternative definitions.

\Lovasz~\cite[p.~23]{LovaszSDP} noted the following formula
relating~$t$ and~$\vartheta$:

\begin{theorem}[\cite{LovaszSDP}]
  \label{thm:t-theta}
  For every graph $G$, we have
  \begin{equation}
    \label{eq:t-theta}
    2 t(G) + 1/\overline{\vartheta}(G) = 1.
  \end{equation}
\end{theorem}

We will show how the relation~\eqref{eq:t-theta} can be used to better
understand some of the properties of the theta number and the
hypersphere number. This will allow us to obtain simpler proofs of
some facts about the theta number and new results about hypersphere
representations.

\subsection{Proof of Theorem~\ref{thm:t-theta}}

We include a proof of Theorem~\ref{thm:t-theta} for the sake of
completeness. We may formulate~$t(G)$ as the SDP
\begin{equation}
  \label{eq:t}
  t(G)
  =
  \min\setst[\big]{
    t
  }{
    \diag(X) = t\ones,\,
    \Lcal_G^*(X) = \ones,\,
    X \in \Psd[V],\, t \in \Reals
  }.
\end{equation}
Here, $\Lcal_G^*$ is the adjoint of the Laplacian~$\Lcal_G$ of~$G$.
The dual of~\eqref{eq:t} is
\begin{equation}
  \label{opt:t-dual}
  \max\setst[\big]{
    \iprodt{\ones}{z}
  }{
    \Diag(y) \succeq \Laplacian{z},\,
    \iprodt{\ones}{y} = 1,\,
    y \in \Reals^V,\,
    z \in \Reals^E
  }.
\end{equation}

Both~\eqref{eq:t} and~\eqref{opt:t-dual} have Slater points, so SDP
strong duality holds for this dual pair of SDPs, i.e., their optimal
values coincide and both optima are attained. In particular, $t(G)$ is
equal to~\eqref{opt:t-dual}. If we write an optimal solution $X^*$
of~\eqref{eq:t} as $X^* = \oprodsym{U}$, then $i \mapsto U^T e_i$ is a
hypersphere representation of~$G$ with squared radius~$t(G)$.
\begin{proof}[Proof of Theorem~\ref{thm:t-theta}]
  We can rewrite the dual~\eqref{opt:t-dual} as
  \begin{equation*}
    t(G)
    =
    \max
    \setst[\big]{
      \myhalf
      \iprod{
        \oprodsym{\ones} - I
      }{
        Y
      }
    }{
      \qform{Y}{\ones} = 1,\,
      Y_{ij} = 0 \text{ } \forall \set{i,j} \in \overline{E}(G),\,
      Y \in \Psd[V]
    }
  \end{equation*}
  by taking $Y \colonequals \Diag(y) - \Lcal_G(z)$. The objective
  value of a feasible solution~$Y$ is $\frac{1}{2}
  \iprod{\oprodsym{\ones}-I}{Y} = \frac{1}{2}\paren{1 - \trace(Y)}$.
  Thus, $t(G) = \myhalf\paren{1-\hat{t}(G)}$, where
  \begin{equation*}
    \hat{t}(G)
    \colonequals
    \min \setst[\big]{
      \trace(Y)
    }{
      \qform{Y}{\ones} = 1,\,
      Y_{ij} = 0 \text{ } \forall \set{i,j} \in \overline{E}(G),\,
      Y \in \Psd[V]
    }.
  \end{equation*}
  It is easy to check that $\hat{t}(G) \overline{\vartheta}(G) = 1$.
\end{proof}

\subsection{Hypersphere and orthonormal representations of graphs}

Let $G = (V,E)$ be a graph. An \textdef{orthonormal representation
  of~$G$} is a function from~$V$ to the unit hypersphere in~$\Reals^d$
for some~$d \geq 1$ that maps non-adjacent nodes to orthogonal
vectors. It is well-known that, if $u \ffrom V \fto \Reals^d$ is a
hypersphere representation of~$G$ with squared radius $t \leq 1/2$,
then the map
\begin{equation}
  \label{eq:ortho-from-sphere}
  q \ffrom i \mapsto
  \sqrt{2}
  \sqbrac[\big]{
    \sqrt{1/2-t}
    \oplus
    u(i)
  }
  \in \Reals \oplus \Reals^d
\end{equation}
is an orthonormal representation of~$\overline{G}$.
Define~$\thetabody(G)$ as the set of all $x \in \Reals_+^V$ such that
$\sum_{i \in V} \paren{\iprodt{c}{p(i)}}^2 x_i \leq 1$ for every
orthonormal representation $p \ffrom V \fto \Reals^d$ of~$G$ and unit
vector $c \in \Reals^d$. Then $\vartheta(G) =
\max\setst{\iprodt{\ones}{x}}{x \in \thetabody(G)}$.

The transformation~\eqref{eq:ortho-from-sphere} allows us to interpret
Theorem~\ref{thm:t-theta} as strong duality for a nonlinear min-max
relation:
\begin{proposition}
  \label{prop:t-minmax}
  Let $G$ be a graph. For every hypersphere representation of~$G$ with
  squared radius~$t$ and every nonzero $x \in
  \thetabody(\overline{G})$, we have \[ 2t + 1/(\iprodt{\ones}{x})
  \geq 1, \] with equality if and only if $t = t(G)$ and
  $\iprodt{\ones}{x} = \overline{\vartheta}(G)$.
\end{proposition}
\begin{proof}
  Set $V \colonequals V(G)$. Let $u \ffrom V \fto \Reals^d$ be a
  hypersphere representation of~$G$ with squared radius~$t$. We may
  assume that $t < 1/2$. Let $x \in \thetabody(\overline{G})$. Define
  an orthonormal representation~$q$ of~$\overline{G}$ from~$p$ as
  in~\eqref{eq:ortho-from-sphere}. Set $c \colonequals 1 \oplus 0 \in
  \Reals \oplus \Reals^d$. Then $(1-2t) \iprodt{\ones}{x} = \sum_{i
    \in V} \paren{\iprodt{c}{q(i)}}^2 x_i \leq 1$.

  The equality case now follows from Theorem~\ref{thm:t-theta}.
\end{proof}

Proposition~\ref{prop:t-minmax} shows that $\overline{\vartheta}(G)$
and elements from~$\thetabody(\overline{G})$ are natural dual objects
for~$t(G)$ and hypersphere representations of~$G$. In fact, using a
well-known description of the elements of $\thetabody(\overline{G})$,
we recover from Proposition~\ref{prop:t-minmax} the following SDP-free
purely geometric min-max relation:
\begin{corollary}
  \label{prop:t-minmax-sdpfree}
  Let $G = (V,E)$ be a graph. For every hypersphere representation
  of~$G$ with squared radius~$t$, every orthonormal representation $p
  \ffrom V \fto \Reals^d$ of~$G$, and every unit vector $c \in
  \Reals^d$ such that $c \not\in p(V)^{\perp}$, we have \[ 2t +
  \sqbrac[\big]{\textstyle\sum_{i \in
      V} \paren{\iprodt{c}{p(i)}}^2}^{-1} \geq 1, \] with equality if
  and only if $t = t(G)$ and $\sum_{i \in
    V} \paren{\iprodt{c}{p(i)}}^2 = \overline{\vartheta}(G)$.
\end{corollary}

\subsection{A Gallai-type identity}

The transformation~\eqref{eq:ortho-from-sphere} may be reversed as
follows. Suppose that $q \ffrom V \fto \Reals^d$ is an orthonormal
representation of~$\overline{G}$ such that, for some positive $\mu \in
\Reals$ and some $u \ffrom V \fto \Reals^{d-1}$, we have
\begin{equation}
  \label{eq:sphere-from-ortho}
  q(i) = \sqrt{2}
  \sqbrac[\big]{
    (2\mu)^{-1/2}
    \oplus
    u(i)
  }
  \qquad \forall i \in V.
\end{equation}
Then $u$ is a hypersphere representation of~$G$ with squared radius
$\myhalf(1-1/\mu)$. We can use~\eqref{eq:ortho-from-sphere}
and~\eqref{eq:sphere-from-ortho} to obtain an identity involving these
objects.
\begin{proposition}
  \label{prop:th-gallai}
  Let $G = (V,E)$ be a graph. Then
  \begin{equation}
    \label{eq:th-gallai}
    2t(G) + \max_{p,c} \min_{i \in V}
    \paren[\big]{\iprodt{c}{p(i)}}^2 = 1,
  \end{equation}
  where $p$ ranges over all orthonormal representations
  of~$\overline{G}$ and $c$ over unit vectors of the appropriate
  dimension.
\end{proposition}
\begin{proof}
  We first prove~``$\leq$'' in~\eqref{eq:th-gallai}. Let $p \ffrom V
  \fto \Reals^d$ be an orthonormal representation of~$\overline{G}$
  and let $c \in \Reals^d$ be a unit vector. We will show that
  \begin{equation}
    \label{eq:th-gallai-leq}
    t(G) \leq \myhalf \paren[\big]{1- \min_{i \in V}
      \paren[\big]{\iprodt{c}{p(i)}}^2}.
  \end{equation}
  It is well-known that there exists an orthonormal representation $q$
  of~$\overline{G}$ and a unit vector~$d$ such that
  $\paren{\iprodt{d}{q(j)}}^2 = \beta \colonequals \min_{i \in
    V} \paren{\iprodt{c}{p(i)}}^2$ for all $j \in V$. If $\beta = 0$,
  then $i \mapsto 2^{-1/2} e_i \in \Reals^V$ shows that $t(G) \leq
  1/2$, so assume that $\beta > 0$. We may assume that $d = e_1$ and
  $\iprodt{d}{q(i)} \geq 0$ for every $i \in V$. Now
  use~\eqref{eq:sphere-from-ortho} with $\mu = 1/\beta$ to get a
  hypersphere representation~$u$ of~$G$ from~$q$ with squared radius
  $\frac{1}{2}(1-\beta)$. This proves~\eqref{eq:th-gallai-leq}.

  Next we prove ``$\geq$'' in~\eqref{eq:th-gallai}. Let $u \ffrom V
  \fto \Reals^d$ be a hypersphere representation of~$G$ with squared
  radius $t(G)$. Build an orthonormal representation~$q$
  of~$\overline{G}$ as in~\eqref{eq:ortho-from-sphere} and pick $c
  \colonequals 1 \oplus 0 \in \Reals \oplus \Reals^d$. Then
  $\paren{\iprodt{c}{q(i)}}^2 = 1-2t(G)$ for every $i \in V$.
\end{proof}

\noindent
(The reciprocal of the second term of the sum on the LHS
of~\eqref{eq:th-gallai} was used as the original definition
of~$\overline{\vartheta}(G)$ by \Lovasz~\cite{Lovasz79a}.)

Note that~\eqref{eq:th-gallai} does not provide a good
characterization of either $t(G)$ or the maximization problem on the
LHS of~\eqref{eq:th-gallai}. In this sense,
Proposition~\ref{prop:th-gallai} is akin to \Gallai's identities for
graphs~\cite[Lemmas~1.0.1 and~1.0.2]{LovaszP86a}.

\subsection{Unit-distance representations in hyperspheres and balls}
\label{sec:balls}

For a graph $G$, let $t_b(G)$ be the square of the smallest radius of
an Euclidean ball that contains a unit-distance representation of~$G$.
This parameter is also mentioned by
\Lovasz~\cite[Proposition~4.1]{LovaszSDP}.

To formulate~$t_b(G)$ as an SDP, replace the constraint~$\diag(X) =
t\ones$ in~\eqref{eq:t} by $\diag(X) \leq t\ones$. The resulting SDP
and its dual have Slater points, so SDP strong duality holds, i.e.,
both optima are attained and the optimal values coincide.

Evidently, $t_b(G) \leq t(G)$ for every graph~$G$. In fact, equality
holds:
\begin{theorem}
  \label{thm:tb-th}
  For every graph~$G$, we have $t_b(G) = t(G)$.
\end{theorem}

If we mimic the proof of Theorem~\ref{thm:t-theta} for $t_b(G)$, we
find that
\begin{equation}
  \label{eq:tb-thetab}
  2t_b(G) + 1/{\overline{\vartheta_b}(G)} = 1,
\end{equation}
where $\vartheta_b(G)$ is defined by adding the constraint $X\ones
\geq 0$ to the SDP~\eqref{eq:theta-sdp}. Thus, by~\eqref{eq:t-theta}
and~\eqref{eq:tb-thetab}, Theorem~\ref{thm:tb-th} is equivalent to the
fact that $\vartheta_b(G) = \vartheta(G)$ for every graph~$G$. This
follows from next result~\cite[Proposition~9]{Gijswijt05a} (this was
pointed out to the first author by Fernando M{\'{a}}rio de Oliveira
Filho):
\begin{proposition}[\cite{Gijswijt05a}]
  \label{prop:gijswijt}
  Let $\mathbb{K} \subseteq \Sym$ be such that $\Diag(h) X \Diag(h)
  \in \mathbb{K}$ whenever $X \in \mathbb{K}$ and $h \in \Reals_+^n$.
  If $\Xh$ is an optimal solution for the optimization problem $ \max
  \setst[\big]{\qform{X}{\ones}}{\trace(X) = 1,\, X \in \mathbb{K}
    \cap \Psd}$, then $\diag(\Xh) = \mu \Xh \ones$ for some positive
  $\mu \in \Reals$.
\end{proposition}

\begin{proof}[Proof of Theorem~\ref{thm:tb-th}]
  Since $\vartheta(G)$ is a relaxation of~$\vartheta_b(G)$, we have
  $\vartheta_b(G) \leq \vartheta(G)$. To prove the reverse inequality,
  let $\Xh$ be an optimal solution for~\eqref{eq:theta-sdp}. By
  Proposition~\ref{prop:gijswijt}, we have $\Xh \ones = \mu^{-1}
  \diag(\Xh) \geq 0$ for some $\mu > 0$. Hence $\Xh$ is feasible for
  the SDP that defines~$\vartheta_b(G)$, whence $\vartheta_b(G) \geq
  \vartheta(G)$.
\end{proof}

\subsection{Hypersphere proofs of $\vartheta$ facts}

The formula~\eqref{eq:t-theta} relating $t(G)$ and
$\overline{\vartheta}(G)$ allows us to infer some basic facts about
the theta number from a geometrically simpler viewpoint.
\begin{theorem}[The Sandwich Theorem~\cite{Lovasz79a}]
  \label{thm:sandwich}
  For any graph~$G$, we have $\omega(G) \leq \overline{\vartheta}(G)
  \leq \chi(G)$.
\end{theorem}
By Theorem~\ref{thm:t-theta} and the fact that
$\overline{\vartheta}(K_n)=n$ for every $n \geq 1$, the Sandwich
Theorem is equivalent to the inequalities $t(K_{\omega(G)}) \leq t(G)
\leq t(K_{\chi(G)})$ for every graph~$G$. The first inequality is
obvious: if $H$ is a subgraph of~$G$, then $t(H) \leq t(G)$. The
second one is also obvious: if $u \ffrom [\ell] \fto \Reals^d$ is a
hypersphere representation of~$K_{\ell}$ and $c \ffrom V(G) \fto
[\ell]$ is a colouring of~$G$, then $u \circ c$ is a hypersphere
representation of~$G$. This hints at a strong connection with graph
homomorphisms, which we will look at more closely in
Section~\ref{sec:hom}.

\Lovasz~\cite[p.~34]{LovaszSDP} mentions that a graph $G$ is bipartite
if and only if $\overline{\vartheta}(G) \leq 2$. The less obvious of
the implications may be easily proved by showing that
$\overline{\vartheta}(C_{n}) > 2$ for every odd cycle~$C_n$. However,
we find that the following proof using hypersphere representations
gives a more enlightening geometric interpretation. By
Theorem~\ref{thm:t-theta}, we must show that $t(G) \leq 1/4$ if and
only if $G$ is bipartite. If $G$ is bipartite, then $G$ has a
hypersphere representation with radius~$1/2$ even in~$\Reals^1$.
Suppose $G$ has a hypersphere representation with radius $\leq 1/2$.
The only pairs of points at distance~$1$ in a hypersphere of
radius~$1/2$ are the pairs of antipodal points, so $G$ is bipartite.

Given graphs $G = (V,E)$ and $H = (W,F)$ with $V \cap W =
\myemptyset$, the \textdef{direct sum of~$G$ and~$H$} is the graph $G
+ H \colonequals (V \cup W, E \cup F)$. It is proved
in~\cite{Knuth94a} that $\overline{\vartheta}(G+H) =
\max\set{\overline{\vartheta}(G),\overline{\vartheta}(H)}$. By
Theorem~\ref{thm:t-theta}, this is equivalent to the geometrically
obvious equation $t(G+H) = \max\set{t(G),t(H)}$. In particular, $t(G)
= \max \setst{t(C)}{C \text{ a component of }G}$. More generally, $
t(G) = \max\setst{t(B)}{B \text{ a block of }G}$. This follows from
the next result, where we denote $G_1 \cup G_2 \colonequals (V_1 \cup
V_2, E_1 \cup E_2)$ and $G_1 \cap G_2 \colonequals (V_1 \cap V_2, E_1
\cap E_2)$ for graphs $G_1 = (V_1,E_1)$ and $G_2 =
(V_2,E_2)$.
\begin{proposition}
  \label{prop:clique-separation}
  Let $G = (V,E)$ be a graph, and suppose $G = G_1 \cup G_2$ for
  graphs $G_1$ and~$G_2$, with $G_1 \cap G_2$ a complete graph. Then
  \[
  t(G) = \max\set{t(G_1),t(G_2)}
  \quad
  \text{and}
  \quad
  \overline{\vartheta}(G) = \max
  \set{\overline{\vartheta}(G_1),
    \overline{\vartheta}(G_2)}.
  \]
\end{proposition}
\begin{proof}
  By Theorem~\ref{thm:t-theta}, it suffices to prove that $t(G) =
  \max\set{t(G_1),t(G_2)}$. Clearly `$\geq$' holds in the desired
  equation. Assume $t(G_1) \geq t(G_2)$. Since the feasible region
  of~\eqref{eq:t} is convex and contains $(\Xb,\bar{t}\,) \colonequals
  \myhalf(I,1)$, there are hypersphere representations~$u$ and~$v$
  of~$G_1$ and~$G_2$, respectively, both with squared radius~$t(G_1)$.
  We may assume that the images of~$u$ and~$v$ live in the same
  Euclidean space. Since $G_1 \cap G_2$ is a complete graph, there is
  an orthogonal matrix~$Q$ such that $Qv(i) = u(i)$ for every $i \in
  V(G_1 \cap G_2)$. If we glue the hypersphere representation $i
  \mapsto Qv(i)$ of~$G_2$ with~$u$, we get a hypersphere
  representation of~$G$ with squared radius~$t(G_1)$.
\end{proof}

This behavior of~$t$ and~$\overline{\vartheta}$ with respect to clique
sums is shared by many other graph parameters, e.g., $\omega$, $\chi$,
the Hadwiger number (the size of the largest clique minor), and the
graph invariant~$\lambda$ introduced in~\cite{HolstLS95a}.

Proposition~\ref{prop:clique-separation} and Theorem~\ref{thm:tb-th}
imply the following purely geometric result:
\begin{corollary}
  Let $G = (V,E)$ be a graph, and suppose $G = G_1 \cup G_2$ for
  graphs $G_1$ and~$G_2$, with $G_1 \cap G_2$ a complete graph. For $i
  \in \set{1,2}$, let $u_i$ be a unit-distance representation of~$G_i$
  contained in an Euclidean ball of radius~$r_i$. Then there is a
  unit-distance representation of~$G$ contained in an Euclidean ball
  of radius $\max\set{r_1,r_2}$.
\end{corollary}
The proof contains an algorithm to build the desired unit-distance
representation of~$G$. However, whereas one would expect such an
algorithm to provide a geometric construction from~$u_1$ and $u_2$,
the one presented essentially needs to solve an SDP, and it may
ignore~$u_1$ and~$u_2$ altogether.

Using basic properties of Laplacians, we can prove the following
behaviour of~$t$ and~$\overline{\vartheta}$ with respect to edge
contraction:
\begin{proposition}
  \label{prop:t-contraction}
  Let $G = (V,E)$ be a graph and let $e = \set{i,j} \in E$. If
  $(\yb,\zb)$ is an optimal solution for~\eqref{opt:t-dual}, then
  $\zb_e \geq t(G) - t(G/e)$. If $\Xb$ is an optimal solution
  for~\eqref{eq:theta-sdp} applied to~$\overline{\vartheta}(G)$, then
  $\overline{\vartheta}(G) \leq
  (2\Xb_{ij}+1)\overline{\vartheta}(G/e)$.
\end{proposition}
\begin{proof}
  See Appendix~\ref{sec:delayed}.
\end{proof}

Finally, using basic properties about the intersection of two
hyperspheres, we can prove a property of~$\overline{\vartheta}$ that
is shared by the parameters~$\omega$, $\chi$, and the fractional
chromatic number~$\chi^*$. The proof is based
on~\cite[Lemma~4.3]{KargerMS98a}.
\begin{proposition}
  \label{prop:nbhood}
  Let $G$ be a graph and $i \in V(G)$ with $N(i) \neq \myemptyset$.
  Then \[
  t(G[N(i)]) \leq 1 - 1/[4t(G)]
  \quad
  \text{and}
  \quad
  \overline{\vartheta}(G) \geq \overline{\vartheta}(G[N(i)])+1. \]
\end{proposition}
\begin{proof}
  See Appendix~\ref{sec:delayed}.
\end{proof}

\subsection{A weighted version}
\label{sec:tw}

For $w \in \Reals_+^V$, define $\vartheta(G,w)$ by replacing the
objective function in~\eqref{eq:theta-sdp} by $\qform{X}{\sqrt{w}}$,
where $(\sqrt{w})_i \colonequals \sqrt{w_i}$ for every $i \in V$. It
is natural to define a weighted hypersphere number~$t(G,w)$ so that it
satisfies a natural generalization of~\eqref{eq:t-theta}, namely,
$2t(G,w) + 1/\vartheta(\overline{G},w) = 1$ whenever $w \neq 0$. By
using the proof of Theorem~\ref{thm:t-theta} as a guide, we arrive at
the definition:
\begin{optproblem}[opt:thw]
  t(G,w) = \min & t & \\
  & \diag(X) = \frac{1}{2} \ones + \paren{t-\frac{1}{2}}w, & \\
  & \Lcal_G^*(X) = \ones + \paren{t-\frac{1}{2}}\Lcal_G^*(W), & \\
  & X \in \Psd[V],\, t \in \Reals. &
\end{optproblem}
This SDP and its dual have Slater points, so SDP strong duality holds.

Even though we cannot offer a nice direct interpretation for this
definition of $t(G,w)$, by construction, we can generalize
Proposition~\ref{prop:t-minmax}:
\begin{theorem}
  \label{thm:thw-minmax}
  Let $G$ be a graph and $w \in \Reals_+^{V(G)} \drop \set{0}$. Then,
  for every feasible solution $(X,t)$ of~\eqref{opt:thw} and every
  nonzero $x \in \thetabody(\overline{G})$, we have \[ 2t +
  1/({\iprodt{w}{x}}) \geq 1, \] with equality if and only if $(X,t)$
  is optimal for~\eqref{opt:thw} and $\iprodt{w}{x} =
  \vartheta(\overline{G},w)$.
\end{theorem}
\begin{proof}
  Set $V \colonequals V(G)$. We may assume that $t < 1/2$. Write $X =
  P^T P$ for some $[d] \times V$ matrix~$P$, and define $p \ffrom V
  \fto \Reals^d$ by $p \ffrom i \mapsto P e_i$. The map $q \ffrom i
  \mapsto \sqrt{2} \sqbrac[\big]{ \sqrt{w_i(1/2-t)} \oplus p(i) } \in
  \Reals \oplus \Reals^d$ is an orthonormal representation
  of~$\overline{G}$. Put $c \colonequals 1 \oplus 0 \in \Reals \oplus
  \Reals^d$. Then $(1-2t) \iprodt{w}{x} = \sum_{i \in
    V} \paren{\iprodt{c}{q(i)}}^2 x_i \leq 1$.

  The equality case now follows by construction.
\end{proof}

If $w \in \Integers_+^V$, it can be shown that $t(G,w) = t(G^w)$,
where $G^w$ is obtained from~$G$ by replacing each node~$i$ by a
clique~$G_i$ on~$w_i$ nodes; if $\set{i,j} \in E(G)$, then every node
in~$G_i$ is adjacent in~$G^w$ to every node in~$G_j$.

In fact, every feasible solution $(\Xb,\bar{t})$ of~\eqref{opt:thw}
encodes a hypersphere representation of~$G^w$ with squared
radius~$\bar{t}$. Indeed, write $\Xb = P^T P$ for some $[d] \times V$
matrix~$P$, and define $p \ffrom i \mapsto Pe_i$. For $i \in V$, let
$q_i \ffrom V(G_i) \fto \Reals^{d_i}$ be a hypersphere representation
of~$G_i$ with squared radius $t(G_i) = \myhalf(1-1/w_i)$. Define $u
\ffrom V(G^w) \fto \Reals^{d} \oplus \paren[\big]{\bigoplus_{i \in V}
  \Reals^{d_i}}$ as follows. For $k \in V(G_i)$, set $u(k)$ to be the
vector whose block in $\Reals^d$ is $w_i^{-1/2}p(i)$ and whose block
in $\Reals^{d_i}$ is $q_i(k)$; all other blocks of $u(k)$ are zero.
Then $u$ is a hypersphere representation of~$G^w$ with squared radius
$\bar{t}$.

\section{Graph homomorphisms and sandwich theorems}
\label{sec:hom}

Let $G$ and $H$ be graphs. A \textdef{homomorphism from~$G$ to~$H$} is
a function $f \ffrom V(G) \fto V(H)$ such that $\set{f(i),f(j)} \in
E(H)$ whenever $\set{i,j} \in E(G)$. If there is a homomorphism from
$G$ to~$H$, we write $G \rightarrow H$.

Note that $t(G) \leq t(H)$ whenever $G \rightarrow H$. Indeed, if $f$
is a homomorphism from~$G$ to~$H$ and~$v$ is a hypersphere
representation of~$H$, then $v \circ f$ is a hypersphere
representation of~$G$. This combines with the graph-theoretic
observation that $K_{\omega(G)} \rightarrow G \rightarrow K_{\chi(G)}$
to yield $t(K_{\omega(G)}) \leq t(G) \leq t(K_{\chi(G)})$, which by
Theorem~\ref{thm:t-theta} is equivalent to the Sandwich
Theorem~\ref{thm:sandwich}.

Motivated by this, we call a real-valued graph invariant~$f$
\textdef{hom-monotone} if $f(G) \leq f(H)$ whenever $G \rightarrow H$
and the following ``nondegeneracy'' condition holds: there is a
non-decreasing function $g \ffrom \Image(f) \fto \Reals$ such that
$g(f(K_n)) = n$ for every integer $n \geq 1$. Using these properties
for an arbitrary graph~$G$ and the fact that $K_{\omega(G)}
\rightarrow G \rightarrow K_{\chi(G)}$, we get $f(K_{\omega(G)}) \leq
f(G) \leq f(K_{\chi(G)})$, and thus
\begin{equation}
  \label{eq:generalized-sandwich}
  \omega(G) \leq g(f(G)) \leq \chi(G).
\end{equation}
(See~\cite{CameronMNSW07a} for a similar use of these ideas.) We point
out that hom-monotonicity cannot recover strong Sandwich Theorems
which state that $\omega(G) \leq \overline{\vartheta}(G) \leq
\chi^*(G)$ since this inequality fails to hold for the hom-monotone
invariant~$\chi$.

The function $g(x) \colonequals 1/(1-2x)$ is non-decreasing on
$[0,1/2) \supseteq \Image(t)$, so $t$ is hom-monotone, and we recover
from~\eqref{eq:generalized-sandwich} the Sandwich
Theorem~\ref{thm:sandwich}.

The reason why~$t$ satisfies the first condition of hom-monotonicity
roughly comes from the fact that the constraints for the
SDP~\eqref{eq:t} of~$t$ are ``uniform'' for the edges, i.e., all edges
are treated in the same way. We are thus led to define other SDPs of
the same type. One such example is the parameter $t_b$. However, as we
have seen in Theorem~\ref{thm:tb-th}, this parameter is equal to $t$.
Now define
\begin{equation}
  \label{eq:t'}
  t'(G)
  \colonequals
  \min\setst[\big]{
    t
  }{
    \diag(X) = t\ones,\,
    \Lcal_G^*(X) \geq \ones,\,
    X \in \Psd[V],\, t \in \Reals
  }.
\end{equation}
Clearly, $t'(G) \leq t(G)$ for every graph~$G$, and it is easy to see
that equality holds if $G$ is node-transitive. In particular, $t'(K_n)
= t(K_n)$ for every $n$. Thus, the function $g(x) \colonequals
1/(1-2x)$ proves that~$t'$ is hom-monotone.

Using~\eqref{eq:generalized-sandwich} and $t'(G) \leq t(G)$, we obtain
$\omega(G) \leq g(t'(G)) \leq g(t(G)) \leq \chi(G)$ for every
graph~$G$. If we mimic the proof of Theorem~\ref{thm:t-theta} for
$t'(G)$, we find that $2t'(G) + 1/\overline{\vartheta'}(G) = 1$, where
$\vartheta'(G)$ is defined by adding the constraint $X \geq 0$
to~\eqref{eq:theta-sdp}, i.e., $g(t'(G)) = \vartheta'(G)$ is the graph
parameter introduced in~\cite{McElieceRR78a} and~\cite{Schrijver79a}.

Let $\dim(G)$ be the minimum $d \geq 0$ such that there is a
unit-distance representation of~$G$ in~$\Reals^d$; consider $\Reals^0
\colonequals \set{0}$. As before, $G \to H$ implies $\dim(G) \leq
\dim(H)$. Since $\dim(K_n) = n-1$, the function $g(x) \colonequals
x+1$ shows that~$\dim$ is hom-monotone, so $\omega(G) \leq \dim(G)+1
\leq \chi(G)$. However, we will see later that computing~$\dim(G)$ is
NP-hard. (A similar parameter was introduced in~\cite{ErdosHT65a}.)

Define~$\dim_h(G)$ similarly as~$\dim(G)$ but for hypersphere
representations of~$G$ with squared radius $\leq 1/2$ and~$\dim_o(G)$
for orthonormal representations of~$\overline{G}$. Such parameters are
also hom-monotone. Clearly $\dim(G) \leq \dim_h(G)$ for every
graph~$G$, but strict inequality occurs for the Mosers spindle (see
Figure~\ref{fig:moser} and the proof of Theorem~\ref{thm:dim-hard}).
Since~\eqref{eq:ortho-from-sphere} shows that $\dim_o(G) \leq
\dim_h(G) + 1$ and \cite{Lovasz79a} shows that
$\overline{\vartheta}(G) \leq \dim_o(G)$, these parameters are related
by $\omega(G) \leq \overline{\vartheta'}(G) \leq
\overline{\vartheta}(G) \leq \dim_o(G) \leq \dim_h(G) + 1 \leq
\chi(G)$. In particular, by~\eqref{eq:t-theta}, we find that
$\dim_h(G) \geq 2t(G)/(1-2t(G))$. Also $\dim_h(G) \leq \chi(G)-1 \leq
\Delta(G)$, where $\Delta(G)$ is the maximum degree of~$G$. In fact,
by Brooks' Theorem, $\dim_h(G) \leq \Delta(G) - 1$ when $G$ is
connected but not complete nor an odd cycle.

\subsection{Hypersphere representations and vector colourings}

The following relaxation of graph colouring was introduced
in~\cite{KargerMS98a}. Let $G = (V,E)$ be a graph. For a real number
$k \geq 1$, a \textdef{vector $k$-colouring of~$G$} is a function $p$
from~$V$ to the unit hypersphere in~$\Reals^d$ for some $d \geq 1$
such that $\iprod{p(i)}{p(j)} \leq -1/(k-1)$ whenever $\set{i,j} \in
E$; we consider the fraction to be $-\infty$ if $k = 1$, so the only
graphs that have a vector $1$-colouring are the graphs with no edges.

A vector $k$-colouring~$p$ of~$G$ is \textdef{strict} if
$\iprod{p(i)}{p(j)} = -1/(k-1)$ for every $\set{i,j} \in E$, and a
strict vector $k$-colouring~$p$ of~$G$ is \textdef{strong} if
$\iprod{p(i)}{p(j)} \geq -1/(k-1)$ whenever $\set{i,j} \in
\overline{E}(G)$.

The \textdef{vector chromatic number of~$G$} is the smallest $k \geq
1$ for which there exists a vector $k$-colouring of~$G$, and the
strict vector chromatic number and strong vector chromatic number are
defined analogously.

It is easy to show (see, e.g., \cite{LaurentR05a}) that the vector
chromatic number of~$G$ is $\overline{\vartheta'}(G)$, the strict
vector chromatic number of~$G$ is $\overline{\vartheta}(G)$, and the
strong vector chromatic number of~$G$ is $\overline{\vartheta^+}(G)$,
known as \Szegedy's number~\cite{Szegedy94a}, where $\vartheta^+(G)$
is defined by replacing the constraints $X_{ij} = 0$ for every
$\set{i,j} \in E$ in~\eqref{eq:theta-sdp} by $X_{ij} \leq 0$ for every
$\set{i,j} \in E$.

Here, we note that a scaling map yields a correspondence between these
variations of vector colourings and unit-distance representations,
provided that the graph~$G$ has at least one
edge.

Let $p$ be a strict vector $k$-colouring of~$G$. Then the map $i
\mapsto tp(i)$, where $t^2 = \frac{1}{2}(1-1/k)$, is a hypersphere
representation of~$G$ with squared radius~$t$. Conversely, if $q$ is a
hypersphere representation of~$G$ with squared radius $t < 1/2$, then
the map $i \mapsto t^{-1/2}q(i)$, is a strict vector $k$-colouring
of~$G$, where $k = 1/(1-2t)$. This correspondence shows that $t(G) =
\frac{1}{2}(1-1/\chi_v(G))$, where $\chi_v(G)$ denotes the strict
vector chromatic number of~$G$.

The same scaling maps as above yield correspondences between vector
$k$\nbd-colourings and the geometric representations arising from the
graph invariant~$t'$, and also between strong vector $k$-colourings
and geometric representations arising from the graph invariant
\begin{optproblem}[opt:t+]
  t^+(G) \colonequals  \min & t & \\
  & \diag(X) = t \ones, & \\
  & X_{ii}-2X_{ij}+X_{jj} = 1, & \forall \set{i,j} \in E(G),\\
  & X_{ii}-2X_{ij}+X_{jj} \leq 1, & \forall \set{i,j} \in
  \overline{E}(G),\\
  & X \in \Psd[V],\, t \in \Reals. &
\end{optproblem}
Note however, that the parameter $t^+$ does not fit into the framework
of hom-monotone graph invariants since the SDP~\eqref{opt:t+} has
non-edge constraints.

We point out here that, while these equivalences between variants of
vector chromatic number and variants of theta number are easy to
prove, they are not as widely known as they should be. For instance,
in~\cite{Bilu06a} it is shown that the vector chromatic
number~$\chi_v'(G)$ of~$G$ satisfies
\begin{equation}
  \label{eq:hoffman}
  \chi_v'(G) \geq \max\setst[\bigg]{1-
    \frac{\lambda_{\max}(B)}{\lambda_{\min}(B)}}{B \in \Acal_G,\,
    B \geq 0},
\end{equation}
where $\lambda_{\max}(\cdot)$ and $\lambda_{\min}(\cdot)$ denote the
largest and smallest eigenvalue, respectively, and $\Acal_G$ denotes
the set of all weighted adjacency matrices of $G = (V,E)$, i.e., all
symmetric $V \times V$ matrices $A$ such that $A_{ij} \neq 0 \implies
\set{i,j} \in E$. However, since $\chi_v'(G) =
\overline{\vartheta'}(G)$, it is possible to adapt the proof of the
Hoffman bounds for~$\vartheta(G)$ (see, e.g.,
\cite[Corollary~33]{Knuth94a}) to show that~\eqref{eq:hoffman}
actually holds with equality.

Also, in~\cite[Remark~3.1]{Meurdesoif05a} it is reported that a
certain graph~$G$ has vector chromatic number strictly smaller than
its strict vector chromatic number, and that it was unknown whether
some such graph existed. However, this statement about the vector
chromatic numbers is equivalent to $\overline{\vartheta'}(G) <
\overline{\vartheta}(G)$, and the existence of graphs satisfying this
strict inequality was already known as far back as~1979
(see~\cite{Schrijver79a}).

We also mention that one of the characterizations of $\vartheta'(G)$
in~\cite{Goemans97a} and~\cite{Galtman00a} is inaccurate. Define an
\textdef{obtuse representation of a graph $G = (V,E)$} to be a map $p
\ffrom V \fto \Reals^d$ for some $d \geq 1$ such that
\begin{enumerate}[(i)]
\item $\norm{p(i)} = 1$ for every $i \in V$, and
\item $\iprod{p(i)}{p(j)} \leq 0$ for every $\set{i,j} \in
  \overline{E}(G)$.
\end{enumerate}
In~\cite[Theorem~1]{Goemans97a} and~\cite[p.~133]{Galtman00a} it is
claimed that
\begin{equation}
  \label{eq:galtman-theta'}
  \vartheta'(G) = \min_{p,c} \max_{i \in V}
  \frac{1}{\paren[\big]{\iprodt{c}{p(i)}}^2},
\end{equation}
where $p$ ranges over obtuse representations of~$G$ and~$c$ ranges
over unit vectors of appropriate dimension. Let $G$ be a $2n$-partite
graph with color classes $C_1,\dotsc,C_{2n}$ such that $\omega(G) =
2n$. Thus, $\vartheta'(\overline{G}) \geq \omega(G) = 2n$. Let $p(j)
\colonequals e_i \in \Reals^n$ for every $j \in C_i$ and $i \in [n]$,
and $p(j) \colonequals -e_i \in \Reals^n$ for every $j \in C_{n+i}$
and $i \in [n]$. Set $c \colonequals n^{-1/2} \ones \in \Reals^n$.
By~\eqref{eq:galtman-theta'}, we get $\vartheta'(\overline{G}) \leq
n$, a contradiction.

Now we show how to fix the formula~\eqref{eq:galtman-theta'}. Given an
obtuse representation $p \ffrom V \fto \Reals^d$ of a graph $G =
(V,E)$, we say that a vector $c \in \Reals^d$ is \textdef{consistent
  with~$p$} if $\iprodt{c}{p(i)} \geq 0$ for every $i \in V$. The next
result is a \Gallai-type identity involving $t'(G)$, parallel to
Proposition~\ref{prop:th-gallai} for $t(G)$.
\begin{proposition}
  \label{prop:t'-gallai}
  Let $G = (V,E)$ be a graph. Then
  \begin{equation}
    \label{eq:t'-gallai}
    2t'(G) + \max_{p,c} \min_{i \in V}
    \paren[\big]{\iprodt{c}{p(i)}}^2 = 1,
  \end{equation}
  where $p$ ranges over all obtuse representations of~$\overline{G}$
  and $c$ over unit vectors consistent with~$p$.
\end{proposition}
\begin{proof}
  This proof is analogous to the proof of
  Proposition~\ref{prop:th-gallai}, with the following slight
  adjustments. In the notation of the proof
  of~\eqref{eq:th-gallai-leq}, the vector~$d$ may be chosen to be
  consistent with the obtuse representation~$q$, so we do not need to
  replace any of the $q(i)$'s by their opposites.
\end{proof}

\begin{corollary}
  Let $G = (V,E)$ be a graph. Then $\vartheta'(G)$ is given
  by~\eqref{eq:galtman-theta'}, where $p$ ranges over obtuse
  representations of~$G$ and $c$ ranges over unit vectors consistent
  with~$p$.
\end{corollary}
\begin{proof}
  This follows from Proposition~\ref{prop:t'-gallai} together with the
  formula $2t'(G) + 1/\overline{\vartheta'}(G) = 1$.
\end{proof}

\section{Unit-distance representations in ellipsoids}

The graph parameter $t_b$ encodes the problem of finding the smallest
Euclidean ball that contains a unit-distance representation of a given
graph. In this section, we study graph parameters that encode the
problem of finding the smallest ellipsoid of a given shape that
contains a unit-distance representation of a given graph.

Let $G = (V,E)$ be a graph. In Section~\ref{sec:balls}, we
defined~$t_b(G)$ as the minimum infinity-norm of the vector
$(\iprodt{u_i}{u_i})_{i \in V}$ over all unit-distance
representations~$u$ of~$G$, where we are using the notation $u_i
\colonequals u(i)$. It is natural to replace the vector
$(\iprodt{u_i}{u_i})_{i \in V}$ in the objective function of the
previous optimization problem with the vector $(\qform{A}{u_i})_{i \in
  V}$ for some fixed $A \in \Pd[d]$. The resulting optimization
problem corresponds to finding the minimum squared radius~$t$ such
that the ellipsoid $\setst{x \in \Reals^d}{\qform{A}{x} \leq t}$
contains a unit-distance representation of~$G$.

We are thus led to define, for every graph $G = (V,E)$, every $A \in
\Psd[d]$ for some $d \geq 1$, and every $p \in [1,\infty]$, the number
$\Ecal_{p}(G;A)$ as the infimum of $\norm[p]{(\qform{A}{u_i})_{i \in
    V}}$ as $u$ ranges over all unit-distance representations of~$G$
in~$\Reals^d$, or equivalently,
\begin{equation}
  \label{eq:ellipse}
  \Ecal_{p}(G;A)
  \colonequals
  \inf \setst[\big]{
    \norm[p]{
      \diag(UAU^T)
    }
  }{
    \Lcal_G^*(\oprodsym{U}) = \ones,\,
    U \in \Reals^{V \times [d]}
  }.
\end{equation}
Note that we allow~$A$ to be singular.

Since the feasible region in~\eqref{eq:ellipse} is invariant under
right-multiplication by matrices in~$\Orth[d]$, we have
$\Ecal_{p}(G;A) = \Ecal_p(G;QAQ^T)$ for every $Q \in \Orth[d]$. In
particular, $\Ecal_p(G;\cdot)$ is a spectral function.

Let us derive some basic properties of the optimal solutions
of~$\Ecal_{p}(G;A)$. First, we prove that if $\Ecal_{p}(G;A)$ is
finite then the corresponding optimal geometric representation exists.
The first observation towards this goal is that, if $G$ is connected,
then the maximum distance between any pair of points in every
unit-distance representation is at most $(\card{V(G)}-1)$.
\begin{theorem}
  \label{thm:e-attains}
  Let $G = (V,E)$ be a graph. Let $A \in \Psd[d]$ for some $d \geq 1$
  and let $p \in [1,\infty]$. If $\Ecal_{p}(G;A) < + \infty$, then
  there exists $U \in \Reals^{V\times[d]}$ such that
  $\Lcal_G^*(\oprodsym{U})=\ones$ and $\norm[p]{\diag(UAU^T)} =
  \Ecal_{p}(G;A)$.
\end{theorem}
\begin{proof}
  We may assume that $G$ is connected. (If not, it suffices to focus
  on the component $H$ of $G$ with $\Ecal_{p}(H;A)=\Ecal_{p}(G;A)$.)
  We may further assume $A = \Diag(a)$ where $a =
  \lambda^{\downarrow}(A) \neq 0$, where $\lambda^{\downarrow}(A)$
  denotes the vector of eigenvalues of $A$, with multiplicities,
  arranged in a nonincreasing order. So, there exists a largest $k \in
  [d]$ so that $a_k \neq 0$. Let $A' \colonequals
  \Diag(a_1,\dotsc,a_k)$. Throughout this proof, let $P \ffrom
  \Reals^d \fto \Reals^k$ denote the projection onto the first $k$
  components, i.e., $P(x_1,\dotsc,x_d)^T = (x_1,\dotsc,x_k)^T$, and
  let $Q \ffrom \Reals^d \fto \Reals^{d-k}$ denote the projection onto
  the last $d-k$ components. Note that $A = P^T A' P$ and $A' \succeq
  a_k I$.

  Let $M \in \Reals$ such that $\Ecal_{p}(G;A) \leq M$. Fix $j \in V$
  arbitrarily. We claim that the following constraints may be added to
  the RHS of~\eqref{eq:ellipse} without changing its optimal value:
  \begin{eqnarray}
    \label{eq:e-attains-1}
    \norm[2]{PU^Te_i}^2 & \leq & B \colonequals (M+1)/a_k
    \qquad
    \text{for every } i \in V,
    \\
    \label{eq:e-attains-2}
    QU^Te_{j} & = & 0.
    \qquad
  \end{eqnarray}
  Let us see why this proves the theorem. Let $U \in
  \Reals^{V\times[d]}$ be feasible for~\eqref{eq:ellipse} and
  satisfy~\eqref{eq:e-attains-1} and~\eqref{eq:e-attains-2}. Let $i
  \in V$ be arbitrary. Since the columns of~$U^T$ form a unit-distance
  representation of~$G$, the distance in~$G$ between $i$ and $j$ is an
  upper bound for $\norm[2]{U^Te_i-U^Te_{j}}$. Hence,
  $\norm[2]{U^Te_i} \leq \norm[2]{U^Te_{j}}+\card{V} =
  \norm[2]{PU^Te_{j}} + \card{V} \leq B^{1/2} + \card{V}$. Thus, the
  new feasible region is compact and we will be done.

  First, we prove that the constraints~\eqref{eq:e-attains-1} may be
  added to~\eqref{eq:ellipse} without changing the optimal value.
  Suppose $U \in \Reals^{V \times [d]}$
  violates~\eqref{eq:e-attains-1} for some $i \in V$. Then
  $\norm[p]{\diag(UAU^T)} \geq e_i^T U A U^T e_i = e_i^T U P^T A' P
  U^T e_i \geq e_i^T U P^T (a_k I) P U^T e_i = a_k \norm[2]{PU^Te_i}^2
  > M+1 \geq \Ecal_p(G;A) + 1$, so $U$ may be discarded from the
  feasible set of~\eqref{eq:e-attains-1}.

  Next, we add the constraint~\eqref{eq:e-attains-2}. Let $U \in
  \Reals^{V \times [d]}$ be feasible for~\eqref{eq:ellipse} and
  satisfy~\eqref{eq:e-attains-1}. Define $X \in \Reals^{V \times [d]}$
  by setting $PX^Te_i \colonequals PU^Te_i$ for every $i \in V$ and
  $QX^Te_i \colonequals QU^Te_i-QU^Te_{j}$ for every $i \in V$. Hence,
  $X$ is feasible for~\eqref{eq:ellipse} and
  satisfies~\eqref{eq:e-attains-1} and~\eqref{eq:e-attains-2}.
  Moreover, $\diag(XAX^T) = \diag(UAU^T)$. This completes the proof.
\end{proof}

A geometrically pleasing, intuitive conjecture is that a suitably
defined notion of a ``centre'' of an optimal representation of every
graph must coincide with the centre of the ellipsoid. The next result
takes a step along this direction by refining the previous theorem.

\begin{theorem}
  \label{thm:e-attains-and-conv}
  Let $G = (V,E)$ be a graph. Let $A \in \Psd[d]$ for some $d \geq 1$
  and let $p \in [1,\infty]$. If $\Ecal_{p}(G;A) < + \infty$, then
  there is a unit-distance representation $u \ffrom V \fto \Reals^d$
  of~$G$ such that $\norm[p]{(\qform{A}{u_i})_{i\in V}} =
  \Ecal_{p}(G;A)$ and $0\in\conv(u(V))$.
\end{theorem}
\begin{proof}
  We use the same assumptions and notation defined in the first
  paragraph of the proof of Theorem~\ref{thm:e-attains}. Let $u \ffrom
  V \fto \Reals^d$ be a feasible solution for $\Ecal_{p}(G;A)$. Let
  $\Ucal$ be the set of all unit-distance representations of~$G$ of
  the form $i \in V \mapsto u_i + r$ for some vector $r \in \Reals^d$
  such that $Pr = 0$. Note that if $k = d$, then $\Ucal$ is a
  singleton. Clearly, every element of~$\Ucal$ has the same objective
  value as~$u$. We will show that if there does not exist some element
  $v \in \Ucal$ such that $0 \in \conv(v(V))$, then $\Ecal_p(G;A) <
  \norm[p]{(\qform{A}{u_i})_{i\in V}}$. Then this theorem will follow
  from Theorem~\ref{thm:e-attains}.

  So, assume that $0 \not\in M \colonequals
  \bigcup_{v\in\Ucal}\conv(v(V))$. Since $M = \conv(u(V)) + \Null(P)$
  is a polyhedron and $0 \not\in M$, there exists $h \in \Reals^d$ and
  $\alpha > 0$ such that $\iprodt{h}{v_i} \geq \alpha$ for every $v
  \in \Ucal$ and $i \in V$. Note that $Qh = 0$ since for each $j \in
  \set{k+1,\dotsc,d}$ the linear function $\iprodt{h}{u_i} + th_j =
  \iprodt{h}{(u_i + te_j)}$ of~$t$ is bounded below by~$\alpha$. Thus,
  \begin{equation}
    \label{eq:e-0-in-conv-sep-hyp}
    \iprodt{h}{u_i} \geq \alpha > 0,\quad \forall i \in V
    \qquad\text{and}\qquad
    h \in \Image(A).
  \end{equation}

  Let $x \in \Reals^d$ such that $Ax = h$ and let $s \colonequals \eps
  x$, where $\eps > 0$ will be chosen later. Define $v \ffrom V \fto
  \Reals^d$ by $v_i \colonequals u_i - s$. Let $i \in V$. Then
  $\qform{A}{v_i} = \qform{A}{u_i} -2\eps \iprodt{h}{u_i} + \eps^2
  \qform{A}{x} $. Hence $\qform{A}{v_i} < \qform{A}{u_i}$ if and only
  if $2\eps \iprodt{h}{u_i} > \eps^2 \qform{A}{x}$. Thus, we will be
  done if we can find $\eps > 0$ such that $2\iprodt{h}{u_i} >
  \eps\qform{A}{x}$. Since $\iprodt{h}{u_i} \geq \alpha > 0$, such
  $\eps$ exists. This shows that, for some choice of $\eps > 0$, we
  have $\qform{A}{v_i} < \qform{A}{u_i}$ for every $i \in V$, whence
  $\Ecal_p(G;A) \leq \norm[p]{(\qform{A}{v_i})_{i\in V}} <
  \norm[p]{(\qform{A}{u_i})_{i\in V}}$.
\end{proof}

The next result shows that it is not very interesting to use
arbitrarily large prescribed embedding dimension~$d$:
\begin{theorem}
  \label{thm:e-lower-dim}
  Let $G = (V,E)$ be a graph. Let $A \in \Psd[d]$ for some $d \geq 1$
  and let $p \in [1,\infty]$. If $k \in [d]$ is such that
  $\Ecal_p(G;A)$ has an optimal solution $u \ffrom V \fto \Reals^d$
  with $\dim(\linspan(u(V))) \leq k$, then
  \begin{equation}
    \label{eq:e-lower-dim}
    \Ecal_p(G;A) =
    \Ecal_p(G;B_k)
  \end{equation}
  where $B_k \colonequals
  \Diag(\lambda_1^{\uparrow}(A),\dotsc,\lambda_{k}^{\uparrow}(A))$. In
  particular, $\Ecal_p(G;A) = \Ecal_p(G;B_{n-1})$ if $d \geq n-1$.
\end{theorem}
\begin{proof}
  We may assume that $A = \Diag(a)$ where $a = \lambda^{\uparrow}(A)$
  (here, $\lambda^{\uparrow}(A)$ denotes the vector of eigenvalues of
  $A$, with multiplicities, arranged in a nondecreasing order). Note
  that $B \colonequals B_k = \Diag(a_1,\dotsc,a_k)$. The proof
  of~`$\leq$' in~\eqref{eq:e-lower-dim} is immediate by appending
  extra zero coordinates to an optimal solution of~$\Ecal_p(G;B)$.

  To prove~`$\geq$,' let $u \ffrom V \fto \Reals^d$ be an optimal
  solution for~$\Ecal_p(G;A)$ such that $\dim(\linspan(u(V))) = k$.
  Then, there exists $Q \in \Orth[d]$ such that, for each $i \in V$,
  the final $d-k$ coordinates of $Qu_i$ are zero. Let $v_i \in
  \Reals^k$ be obtained from~$Qu_i$ by dropping the final $d-k$ (zero)
  coordinates. If $C \in \Psd[k]$ is the principal submatrix
  of~$QAQ^T$ indexed by $[k]$, then $(\qform{C}{v_i})_{i \in V} =
  (\qform{A}{u_i})_{i \in V}$. Hence, $\Ecal_p(G;A) =
  \norm[p]{(\qform{A}{u_i})_{i\in V}} =
  \norm[p]{(\qform{C}{v_i})_{i\in V}} \geq \Ecal_p(G;C)$. By
  interlacing of eigenvalues, $\lambda^{\uparrow}(C) \geq
  \lambda^{\uparrow}(B)$. Hence, $\Ecal_p(G;A) \geq \Ecal_p(G;C) \geq
  \Ecal_p(G;B)$.

  It follows from Theorem~\ref{thm:e-attains-and-conv} that
  $\Ecal_p(G;A) = \Ecal_p(G;B_{n-1})$ if $d \geq n-1$.
\end{proof}

It is clear that $\Ecal_{p}(G;A) = 0$ if and only if $\dim(G) \leq
\dim(\Null(A))$. So deciding whether $\dim(G) \leq k$ for any
fixed~$k$ reduces to computing $\Ecal_{p}(G;A)$ for any $p \in
[1,\infty]$ where $A$ is a matrix of nullity~$k$. It is easy to see
that the former decision problem is NP-hard
(see~\cite[Theorem~4]{HorvatKP11a}). We give below a shorter proof.
\begin{theorem}[\cite{HorvatKP11a}]
  \label{thm:dim-hard}
  The problem of deciding whether $\dim(G) \leq 2$ for a given input
  graph~$G$ is NP-hard.
\end{theorem}
\begin{proof}
  Let $k$ be a fixed positive integer. \Saxe~\cite{Saxe80a} showed
  that the following problem is NP-hard: given an input graph $G =
  (V,E)$ and $\ell \ffrom E \fto \Reals_+$, decide whether there
  exists $u \ffrom V \fto \Reals^k$ such that
  $\norm{u(i)-u(j)}=\ell_{\set{i,j}}$ for every $\set{i,j} \in E$.
  \Saxe{} showed that the problem remains NP-hard even if we require
  $\ell(E) \subseteq \set{1,2}$.

  We will show a polynomial-time reduction from the above problem with
  $k=2$ and $\ell(E) \subseteq \set{1,2}$ to the problem of deciding
  whether $\dim(G) \leq 2$. It suffices to show how we can replace any
  edge of the input graph~$G$ which is required to be embedded as a
  line segment of length 2 by some gadget graph~$H$ so that every
  unit-distance representation of~$H$ in~$\Reals^2$ maps two specified
  nodes of~$H$ to points at distance~$2$.

  \begin{figure}
    \centering
    \includegraphics{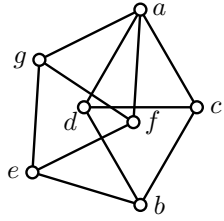}
    \caption{The Mosers spindle; see~\cite{Soifer09a}.}
    \label{fig:moser}
  \end{figure}
  Consider the graph~$M$ known as the Mosers spindle shown in
  Figure~\ref{fig:moser}. The subgraph of~$M$ induced by
  $\set{a,b,c,d}$ has exactly two unit-distance representations
  in~$\Reals^2$ modulo rigid motions: one of them as displayed in
  Figure~\ref{fig:moser}, and the other one maps nodes~$a$ and~$b$ to
  the same point. We claim that, in any unit-distance
  representation~$u$ of~$M$ in~$\Reals^2$, the nodes~$a$ and~$b$ are
  not mapped to the same point. Suppose otherwise. Since the points
  $u(e), u(f), u(g)$ are at distance~$1$ from $u(a) = u(b)$ and from
  each other, $u$ shows that $\dim(K_4) \leq 2$, whereas clearly
  $\dim(K_4) \geq 3$.

  \begin{figure}
    \centering
    \includegraphics{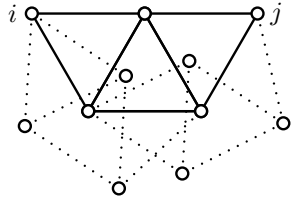}
    \caption{The gadget graph~$H$.}
    \label{fig:gadget}
  \end{figure}
  Let $H$ be the gadget shown in Figure~\ref{fig:gadget}, which
  consists of two copies of~$M$ sharing a triangle (some edges of~$M$
  are drawn in dots for the sake of ease of visualization). Then,
  every unit-distance representation of~$H$ in~$\Reals^2$ maps the
  nodes~$i$ and~$j$ to points at distance~$2$. Thus, if we replace the
  corresponding edges~$\set{i,j}$ of the input graph~$G$ by~$H$, we
  obtain a graph~$G'$ such that $\dim(G') \leq 2$ if and only if $G$
  can be embedded in~$\Reals^2$ with the prescribed edge lengths.
\end{proof}

It follows from Theorem~\ref{thm:dim-hard} that, for any fixed $p \in
[1,\infty]$, the problem of computing $\Ecal_{p}(G;A)$ for an input
graph~$G$ and $A \in \Psd[V(G)]$ is NP-hard. Hence the graph
parameter~$t_b = t$ is in a sense on the borderline of tractability.

\subsection{The extreme cases $p \in \set{1,\infty}$}

For every matrix $U \in \Reals^{V \times V}$, if we set $X
\colonequals \oprodsym{U}$, then there exists an orthogonal $V \times
V$ matrix~$Q$ such that $U^T = QX^{1/2}$. Hence, if $A \in \Psd[V]$,
then
\begin{optproblem}[opt:ellipse-ortho]
  \Ecal_{p}(G;A)
  =
  \inf &
  \norm[p]{\diag(X^{1/2}Q^TAQX^{1/2})} & \\
  & \Lcal_G^*(X) = \ones & \\
  & X \in \Psd[V],\, Q \in \Orth[V]. &
\end{optproblem}
When $p = 1$, the objective function in~\eqref{opt:ellipse-ortho} is
$\trace\paren{Q^TAQX} = \iprod{Q^TAQ}{X}$ so we can write
\begin{equation}
  \label{eq:ellipse1-sdp}
  \Ecal_{1}(G;A)
  =
  \inf_{Q \in \Orth[V]}
  t_{Q^TAQ}(G)
\end{equation}
where $t_W(G)$ is defined for any $W \in \Sym[V]$ as the SDP
\begin{equation}
  \label{eq:tW-sdp}
  t_W(G)
  \colonequals
  \inf\setst[\big]{\iprod{W}{X}}{\Lcal_G^*(X) = \ones,\, X \in \Psd[V]}.
\end{equation}

\begin{proposition}
  \label{prop:tW-finite}
  Let $G = (V,E)$ be a connected graph and let $W \in \Sym[V]$. Then
  $t_W(G)$ is finite if and only if $\qform{W}{\ones} > 0$ or $W\ones
  = 0$. Moreover, whenever $t_W(G)$ is finite, both~\eqref{eq:tW-sdp}
  and its dual SDP have optimal solutions and their optimal values
  coincide.
\end{proposition}

The parameter~$t_W(G)$ thus underlies the parameters $\Ecal_{1}(G;A)$
as well as the hypersphere number~$t(G)$, since~\eqref{opt:t-dual}
shows that \[ t(G) = \min\setst{t_{\Diag(y)}(G)}{\iprodt{\ones}{y} =
  1,\, y \in \Reals^V}. \]

If $X$ is feasible in~\eqref{eq:tW-sdp} for $G = K_n$, then~$X$ is
completely determined by its diagonal entries. Using this fact, it is
easy to prove that the feasible region of~\eqref{eq:tW-sdp} for $G =
K_n$ is
\begin{multline}
  \label{eq:unit-dist-Kn}
  \setst{X \in \Psd}{\Lcal_{K_n}^*(X)=\ones}
  \\
  =
  \setst{(y\ones^T+\ones y^T+2I)/4}{\norm{\ones}\norm{y} \leq
    \iprodt{\ones}{y}+2,\,y\in \Reals^n}.
\end{multline}
Using a second-order cone programming formulation, we can show that
\begin{equation}
  \label{eq:tW-Kn}
  2t_W(K_n)
  =
  \begin{cases}
    \displaystyle
    \trace(W)
    - \frac{\norm{W\ones}^2}{\qform{W}{\ones}}
    & \text{if }\qform{W}{\ones} > 0 \\
    \trace(W) & \text{if } W\ones = 0 \\
    - \infty & \text{otherwise}.
  \end{cases}
\end{equation}

Let us use~\eqref{eq:ellipse1-sdp} and~\eqref{eq:tW-Kn} to compute
$\Ecal_1(G;A)$. Let $A \in \Psd$ be nonzero. Since $Q\ones \not\in
\Null(A)$ for some $Q \in \Orth[n]$, it follows
from~\eqref{eq:ellipse1-sdp} and~\eqref{eq:tW-Kn} that \[
2\Ecal_{1}(K_n;A) = \trace(A) -
\sup\setst[\Big]{\frac{\norm{Q^TAQ\ones}^2}{\qform{\qform{A}{Q}}{\ones}}}{Q\ones
  \not\in \Null(A),\, Q \in \Orth[n]}. \] The supremum may be replaced
by $\sup\setst{(\qform{A^2}{h})/(\qform{A}{h})}{h \in
  \Null(A)^{\perp}}$, which is easily seen to be $\lambda_{\max}(A)$.
This implies with Theorem~\ref{thm:e-lower-dim} that
\begin{equation}
  \label{eq:ellipse1-Kn}
  \Ecal_{1}(K_n;A)
  =
  \begin{cases}
    \myhalf\sum_{i=1}^{n-1}\lambda_i^{\uparrow}(A)
    &
    \text{if }A \in \Psd[d] \text{ with }d \geq n-1
    \\
    +\infty
    &
    \text{otherwise}.
  \end{cases}
\end{equation}

For the other extreme $p = \infty$, the first property of
hom-monotonicity holds. More precisely, let $(a_n)_{n \in
  \Integers_{++}}$ be a nondecreasing sequence of positive reals.
Define $A_n \colonequals \Diag(a_1,\dotsc,a_n)$ for every $n \in
\Integers_{++}$. Then,
\begin{equation}
  G \to H \implies \Ecal_{\infty}(G;A_n) \leq \Ecal_{\infty}(H;A_n).
\end{equation}
We do not know whether the invariant $\Ecal_{\infty}$ satisfies the
second property of hom-monotonicity. In fact, we do not know an
analytical formula to compute~$\Ecal_{\infty}(K_n;A)$ in terms of~$A$.
However, we have such a formula for an infinite family of complete
graphs, as we now describe. Let $H$ be an $n \times n$ Hadamard
matrix, i.e., $H$ is $\set{\pm 1}$-valued and $H^T H = nI$. We may
assume that $H$ has the form $H^T =
\begin{bmatrix}
  \ones & L^T
\end{bmatrix}.$ Then $L^T L = nI - \oprodsym{\ones}$, so
$\frac{1}{2n}\Lcal_{K_n}^*(L^T L) = \ones$, i.e., the map $i \mapsto
(2n)^{-1/2}Le_i$ is a unit-distance representation of~$K_n$. This map
is called a \textdef{Hadamard representation of~$K_n$}.

\begin{theorem}
  Let $n \in \Integers_{++}$ such that there exists an $n \times n$
  Hadamard matrix. Then, for any $p \in [1,\infty]$ and diagonal $A
  \in \Psd[n-1]$, every Hadamard representation of~$K_n$ is an optimal
  solution for~$\Ecal_p(K_n;A)$.
\end{theorem}
\begin{proof}
  The objective value of the Hadamard representation~$\bar{L}$
  of~$K_n$ in the optimization problem~$\Ecal_p(K_n;A)$ is
  $\sqbrac[\big]{\frac{\trace(A)}{2n}} \norm[p]{\ones}$. Thus,
  $\bar{L}$ is optimal for~$p=1$ by~\eqref{eq:ellipse1-Kn}. From the
  inequality $\norm[1]{x} \leq n\norm[\infty]{x}$ we get
  $\Ecal_{\infty}(K_n;A) \geq \frac{1}{n} \Ecal_1(K_n;A)$, which shows
  that~$\bar{L}$ is optimal for $p = \infty$. Therefore, $\bar{L}$ is
  optimal for every $p \in [1,\infty]$.
\end{proof}

It is natural to lift a Hadamard representation~$h$ of~$K_n$ to obtain
a frugal feasible solution for~$\Ecal(K_{n+1};A)$. The image of~$h$ is
an $(n-1)$-dimensional simplex~$\Delta$. If~$v$ is a vertex of an
$n$-dimensional simplex whose opposite facet is~$\Delta$, then the
line segment~$L$ joining~$v$ to the barycenter of~$\Delta$ is the
shortest line segment joining~$v$ to~$\Delta$. It makes sense to
align~$L$ with the most expensive axis, i.e., the one corresponding
to~$\lambda_{\max}(A)$. Suppose $A = \Diag(a)$ and $\norm[\infty]{a} =
a_n$. We thus obtain a unit-distance representation $u$ of~$K_{n+1}$
in~$\Reals^n$ of the form
\begin{displaymath}
  u(i)
  \colonequals
  \begin{cases}
    h(i) \oplus \alpha, & \text{if }i \in [n] \\
    0 \oplus \sqbrac[\big]{
      \alpha + \paren[\big]{
        \frac{
          n+1
        }{
          2n
        }
      }^{\scriptscriptstyle 1/2}
    }, & \text{if }i = n+1.
  \end{cases}
\end{displaymath}
By optimizing the shift parameter~$\alpha$, we obtain the following
upper bound:
\begin{proposition}
  Let $n \in \Integers_{++}$ such that there exists an $n \times n$
  Hadamard matrix. If $A \in \Psd$, then
  \begin{equation}
    \label{eq:hadamard-plus-one}
    \Ecal_{\infty}(K_{n+1};A)
    \leq
    \frac{
      \trace(A)
    }{
      2(n+1)
    }
    +
    \frac{
      \paren[\big]{
        \trace(A) - n\lambda_{\max}(A)
      }^2
    }{
      8n(n+1) \lambda_{\max}(A)
    }.
  \end{equation}
  Equality holds for $n=2$ if $A \succ 0$.
\end{proposition}

The proof of equality for $n=2$ involves the obvious parametrization
of~$\Orth[2]$ and basic trigonometry.

\appendix

\section{Proofs of Propositions~\ref{prop:t-contraction}
  and~\ref{prop:nbhood} and Equations~\eqref{eq:unit-dist-Kn}
  and~\eqref{eq:tW-Kn}}
\label{sec:delayed}

\begin{proof}[Proof of Proposition~\ref{prop:t-contraction}]
  Let $(\yb,\zb)$ be an optimal solution for~\eqref{opt:t-dual}. We
  will construct a feasible solution for~\eqref{opt:t-dual} applied to
  $G/e$ with objective value $t(G) - \zb_e$. Assume $e = \set{a,b}$
  and $V' \colonequals V(G/e) = V \drop \set{b}$, so we are denoting
  the contracted node of~$G/e$ by~$a$. Let $P$ be the $V' \times V$
  matrix defined by $P \colonequals e_a e_b^T + \sum_{i \in V'}
  \oprodsym{e_i}$. Then $P \Lcal_G(\zb) P^T = \Lcal_{G/e}(\zh)$, where
  $\zh \in \Reals^{E(G/e)}$ is obtained from $\zb$ as follows. In
  taking the contraction $G/e$ from~$G$, immediately after we identify
  the ends of~$e$, but before we remove resulting parallel edges,
  there are at most two edges between each pair of nodes of $G/e$, as
  we assume that $G$ is simple. If there is exactly one edge between
  nodes $i$ and~$j$, we just set $\zh_{\set{i,j}} \colonequals
  \zb_{\set{i,j}}$. If there are two edges joining nodes~$i$ and~$j$,
  say~$f$ and~$f'$, we put $\zh_{\set{i,j}} \colonequals \zb_f +
  \zb_{f'}$.

  Similarly, if we define $\yh \ffrom V' \fto \Reals$ by putting
  $\yh_i \colonequals \yb_i$ for $i \in V' \drop \set{a}$ and $\yh_a
  \colonequals \yb_a + \yb_b$, then $P \Diag(\yb) P^T = \Diag(\yh)$.
  Since $P\, \Psd[V]\, P^T \subseteq \Psd[V']$, we see that
  $(\yh,\zh)$ is a feasible solution for~\eqref{opt:t-dual} applied
  to~$G/e$, and its objective value is $\zh(E(G/e)) = \zb(E) - \zb_e$.

  To prove the inequality involving~$\overline{\vartheta}(G)$,
  use~\eqref{eq:t-theta} together with its proof to see that $\Xb$
  corresponds to an optimal solution $(\yb,\zb)$
  for~\eqref{opt:t-dual} with $\Xb/\overline{\vartheta}(G) =
  \Diag(\yb) - \Lcal_G(\zb)$, so $\yb_e =
  \Xb_{ij}/\overline{\vartheta}(G)$.
\end{proof}

\begin{proof}[Proof of Proposition~\ref{prop:nbhood}]
  By Theorem~\ref{thm:t-theta}, it suffices to show $t(G[N(i)]) \leq
  1-1/\sqbrac{4t(G)}$. Let $p \ffrom V \fto \Reals^d$ be a hypersphere
  representation of~$G$ with squared radius $t \colonequals t(G)$. We
  may assume that $p(i) = t^{1/2} e_1$. For every $j \in N(i)$, we
  have $1 = \norm{p(i)-p(j)}^2 = \norm{p(i)}^2 + \norm{p(j)}^2 -
  2\iprod{p(i)}{p(j)} = 2t - 2t^{1/2} [p(j)]_1$. Hence, $[p(j)]_1 =
  (2t-1)/(2t^{1/2}) \equalscolon \beta$ for every $j \in N(i)$. Define
  the following hypersphere representation of $G[N(i)]$: for each $j
  \in N(i)$, let $q(j)$ be obtained from $p(j)$ by dropping the first
  coordinate. The squared radius of the resulting hypersphere
  representation is $t-\beta^2 = 1-1/(4t)$.
\end{proof}

\begin{proof}[Proof of~\eqref{eq:unit-dist-Kn}]
  Let $X \in \Sym[V]$. Then $\Lcal_{K_n}^*(X) = \ones$ if and only if
  $4X = y\ones^T + \ones y^T + 2I$ for some $y \in \Reals^V$; for the
  `only if' part, use $y \colonequals 2\diag(X) - \ones$.

  Let $y \in \Reals^V$. The smallest eigenvalue of $y\ones^T + \ones
  y^T$ is $\iprodt{\ones}{y} - \norm{\ones}\norm{y}$. Thus, $y\ones^T
  + \ones y^T + 2I \succeq 0$ if and only if $\norm{\ones} \norm{y}
  \leq \iprodt{\ones}{y} + 2$.
\end{proof}

\begin{proof}[Proof of \eqref{eq:tW-Kn}]
  Assume first that $W = \Diag(w)$ for some $w \in \Reals^n$. By
  Proposition~\ref{prop:tW-finite}, finiteness of $t_W(K_n)$ implies
  $\iprodt{\ones}{w} > 0$ or $w = 0$. Assume the former.
  By~\eqref{eq:unit-dist-Kn},
  \begin{equation}
    \label{eq:tW-Kn-SOC}
    2t_W(K_n)
    =
    \iprodt{\ones}{w}
    +
    \min\setst{
      \iprodt{w}{y}
    }{
      \norm{\ones}y_0
      -
      \iprodt{\ones}{y}
      =
      2,\,
      y_0 \oplus y \in \soc{n}
    },
  \end{equation}
  where $\soc{n} \colonequals \setst{y_0 \oplus y \in \Reals \oplus
    \Reals^n}{\norm{y} \leq y_0}$. The second-order cone program on
  the RHS of~\eqref{eq:tW-Kn-SOC} has $\yb_0 \oplus \yb \colonequals
  (2 + \norm{\ones}^2)/\norm{\ones} \oplus \ones$ as a Slater point,
  and its dual is $\max\setst{2\mu}{-\mu\norm{\ones} \oplus
    (w+\mu\ones) \in \soc{n},\, \mu \in \Reals}$. Since $\mu^*
  \colonequals -\norm{w}^2/(2\iprodt{\ones}{w})$ is optimal for the
  dual, it follows that
  \begin{equation}
    \label{eq:tW-Kn-W-diag}
    2t_{\Diag(w)}(K_n)
    =
    \begin{cases}
      \iprodt{\ones}{w}
      -
      \norm{w}^2
      /
      (\iprodt{\ones}{w})
      &
      \text{if }\iprodt{\ones}{w} > 0
      \\
      0 & \text{if }w = 0
      \\
      - \infty & \text{otherwise}.
    \end{cases}
  \end{equation}

  Now we drop the diagonal assumption, so let $W \in \Sym$ such that
  $\qform{W}{\ones} > 0$ or $W\ones = 0$. For $y \in \Reals^n$, we can
  write $ \iprod{W}{y\ones^T + \ones y^T} = \iprod{W\ones}{2y} =
  \iprod{\Diag(W\ones)}{y\ones^T + \ones y^T}$, so
  $\iprod{W}{y\ones^T+\ones y^T +2I} = \iprod{\Diag(W\ones)}{y\ones^T
    + \ones y^T + 2I} - 2\qform{W}{\ones} + 2\trace(W)$, i.e., \[
  4t_W(K_n) = 4t_{\Diag(W\ones)}(K_n) - 2\qform{W}{\ones} +
  2\trace(W) \] by~\eqref{eq:unit-dist-Kn}. Hence, \eqref{eq:tW-Kn}
  follows from~\eqref{eq:tW-Kn-W-diag}.
\end{proof}

\section{Proofs for the sake of completeness}

\begin{proof}[Proof of Proposition~\ref{prop:gijswijt}]
  If $\Xh_{ii} = 0$ for some $i \in V$, then $\Xh e_i = 0$ and we are
  done by induction on~$n$. So we may assume that $\diag(\Xh) > 0$.
  Define $x \in \Reals^n$ by $x_i \colonequals \Xh_{ii}^{1/2}$ and let
  $\Xb \colonequals \Diag(x)^{-1} \Xh \Diag(x)^{-1}$. Note that $\Xb
  \in \mathbb{K} \cap \Psd$ and $\diag(\Xb) = \ones$.

  For every $h \in \Reals_+^n$ with $\norm{h} = 1$, the matrix
  $\Diag(h) \Xb \Diag(h)$ is feasible in the optimization problem with
  objective value $\qform{\Xb}{h}$. Since $\Xh = \Diag(x) \Xb
  \Diag(x)$ is an optimal solution, we see that $h = x$ attains the
  maximum of $\qform{\Xb}{h}$ over all $h \in \Reals_+^n$ with
  $\norm{h} = 1$. Since $x > 0$, then $h = x$ attains the maximum
  of~$\qform{\Xb}{h}$ also over all $h \in \Reals^n$ with $\norm{h} =
  1$. Thus, for $\lambda \colonequals \lambda_{\max}(\Xb)$, we have
  $\Xb x = \lambda x$, so $\Xh \ones = \Diag(x) \Xb \Diag(x) \ones =
  \Diag(x) \Xb x = \lambda \Diag(x) x = \lambda \diag(\Xh)$.
\end{proof}

\begin{proof}[Proof of~\eqref{opt:ellipse-ortho}]
  If $(X,Q)$ is a feasible solution for the RHS
  of~\eqref{opt:ellipse-ortho}, then $U^T \colonequals Q X^{1/2}$ is
  feasible in~\eqref{eq:ellipse} and has objective value
  $\norm[p]{\diag(UAU^T)} = \norm[p]{\diag(X^{1/2}Q^TAQX^{1/2})}$,
  which is the objective value of~$(X,Q)$ in the RHS
  of~\eqref{opt:ellipse-ortho}.

  Let $U$ be a feasible solution for~\eqref{eq:ellipse}. Let $X
  \colonequals \oprodsym{U}$. Then $X^{1/2} = QU^T$ for some $Q \in
  \Orth[V]$. The objective value of~$(X,Q^T)$ in the RHS
  of~\eqref{opt:ellipse-ortho} is
  $\norm[p]{\diag(X^{1/2}QAQ^TX^{1/2})} = \norm[p]{\diag(UAU^T)}$,
  which is the objective value of~$U$ in~\eqref{eq:ellipse}.
\end{proof}

\subsection{Proof of Proposition~\ref{prop:tW-finite}}

\begin{proposition}
  \label{prop:tW-Slater}
  Let $G = (V,E)$ be a connected graph and let $W \in \Sym[V]$. Then
  there exists $z \in \Reals^E$ such that $\Lcal_G(z) \prec W$ if and
  only if $\qform{W}{\ones} > 0$.
\end{proposition}
\begin{proof}
  If $W \succ \Lcal_G(z)$ for some $z \in \Reals^E$, then
  $\qform{W}{\ones} = \qform{(W-\Lcal_G(z))}{\ones} > 0$.

  Suppose that $\qform{W}{\ones} > 0$. Let $L := \Lcal_G(\ones)$ and
  assume $V = [n]$. Let $Q \in \Orth[n]$ such that $Qe_1 = n^{-1/2}
  \ones$. Then $Q^T L Q = 0 \oplus L'$ for some $L' \in \Pd[n-1]$,
  since $G$. Let $A \in \Sym[n-1]$, $b \in \Reals^{n-1}$ and $\gamma
  \in \Reals$ such that
  \[
  Q^T W Q
  =
  \begin{bmatrix}
    \gamma & b^T \\
    b & A
  \end{bmatrix}.
  \]
  Note that $\gamma = e_1^T Q^T W Q e_1 = n^{-1}\qform{W}{\ones} > 0$.
  Thus, for every $\lambda \in \Reals$, we have $Q^T(W-\lambda L)Q
  \succ 0$ if and only if $ A- \lambda L' - \gamma^{-1}\oprodsym{b}
  \succ 0$. Since $L' \succ 0$, we know that for $\lambda$ negative
  and with sufficiently large magnitude, we have $Q^T(W-\lambda L)Q
  \succ 0$, and hence $W \succ \Lcal_G(\lambda \ones)$.
\end{proof}

\begin{proposition}
  \label{prop:tW-bounded-below}
  Let $G$ be a graph and let $W \in \Sym[V(G)]$ such that
  $\qform{W}{\ones} = 0$. If $t_W(G) > -\infty$, then $\ones \in
  \Null(W)$.
\end{proposition}
\begin{proof}
  Assume $t_W(G) > -\infty$. Since~\eqref{eq:tW-sdp} has
  $\frac{1}{2}I$ as a Slater point, the dual
  \begin{equation}
    \label{eq:tW-dual}
    \sup\setst[\big]{
      \iprodt{\ones}{z}
    }{
      W \succeq \Lcal_G(z),\,
      z \in \Reals^E
    }.
  \end{equation}
  of~\eqref{eq:tW-sdp} has an optimal solution~$z$. Assume $V = [n]$
  and let $Q \in \Orth[n]$ such that $Qe_1 = n^{-1/2}\ones$. Then
  $Q^T(W-\Lcal_G(z))Q \succeq 0$ and $e_1^T Q^T(W-\Lcal_G(z))Q e_1 =
  n^{-1}\qform{(W-\Lcal_G(z))}{\ones} = 0$ imply that
  \[ e_k^TQ^TW\ones = e_k^TQ^T(W-\Lcal_G(z))\ones = n^{1/2} e_k^T
  Q^T(W-\Lcal_G(z))Qe_1 = 0 \] for every $k \in [n]$. Thus, $W \ones \in
  \set{Qe_2,\dotsc,Qe_n}^{\perp} = \set{\ones}^{\perp\perp}$, which
  together with $\qform{W}{\ones} = 0$ implies $W\ones = 0$.
\end{proof}

\begin{proposition}
  \label{prop:tW-ones-in-NullW}
  Let $G$ be a connected graph and let $W \in \Sym[V(G)]$ such that
  $W\ones = 0$. Then~\eqref{eq:tW-sdp} and~\eqref{eq:tW-dual} have
  optimal solutions and the optimal values coincide.
\end{proposition}
\begin{proof}
  Since $W \ones = 0$, it is easy to check that the constraint
  $\qform{X}{\ones} = 0$ may be added to~\eqref{eq:tW-sdp} without
  changing the optimal value. The dual of this augmented SDP is $
  \sup\setst[\big]{ \iprodt{\ones}{z} }{ W - \mu\oprodsym{\ones}
    \succeq \Lcal_G(z),\, z \in \Reals^E,\, \mu \in \Reals }$. By
  Proposition~\ref{prop:tW-Slater}, this dual has a Slater point
  $(z,\mu)$ with $\mu = 1$, so~\eqref{eq:tW-sdp} has an optimal
  solution. Since~\eqref{eq:tW-sdp} has a Slater point and is bounded
  below, its dual~\eqref{eq:tW-dual} has an optimal solution and the
  optimal values coincide.
\end{proof}

\begin{proof}[Proof of Proposition~\ref{prop:tW-finite}]
  If $\qform{W}{\ones} > 0$, then \eqref{eq:tW-sdp} and its
  dual~\eqref{eq:tW-dual} have Slater points by
  Proposition~\ref{prop:tW-Slater}. If $\qform{W}{\ones} < 0$, then
  $X_t \colonequals \myhalf I + t\oprodsym{\ones}$ with $t \to \infty$
  shows that $t_W(G) = -\infty$. Assume now that $\qform{W}{\ones} =
  0$. If $W\ones \neq 0$, then $t_W(G) = -\infty$. Otherwise, apply
  Proposition~\ref{prop:tW-ones-in-NullW}.
\end{proof}


\begin{thebibliography}{10}

\bibitem{Bilu06a}
Y.~Bilu.
\newblock Tales of {H}offman: three extensions of {H}offman's bound on the
  graph chromatic number.
\newblock {\em J. Combin. Theory Ser. B}, 96(4):608--613, 2006.

\bibitem{CameronMNSW07a}
P.~J. Cameron, A.~Montanaro, M.~W. Newman, S.~Severini, and A.~Winter.
\newblock On the quantum chromatic number of a graph.
\newblock {\em Electron. J. Combin.}, 14(1):Research Paper 81, 15 pp.
  (electronic), 2007.

\bibitem{Chung97a}
F.~R.~K. Chung.
\newblock {\em Spectral graph theory}, volume~92 of {\em CBMS Regional
  Conference Series in Mathematics}.
\newblock Published for the Conference Board of the Mathematical Sciences,
  Washington, DC, 1997.

\bibitem{ErdosHT65a}
P.~Erd{\H{o}}s, F.~Harary, and W.~T. Tutte.
\newblock On the dimension of a graph.
\newblock {\em Mathematika}, 12:118--122, 1965.

\bibitem{Galtman00a}
A.~Galtman.
\newblock Spectral characterizations of the {L}ov\'asz number and the
  {D}elsarte number of a graph.
\newblock {\em J. Algebraic Combin.}, 12(2):131--143, 2000.

\bibitem{Gijswijt05a}
D.~Gijswijt.
\newblock {\em Matrix Algebras and Semidefinite Programming Techniques for
  Codes}.
\newblock PhD thesis, University of Amsterdam, 2005.

\bibitem{Goemans97a}
M.~X. Goemans.
\newblock Semidefinite programming in combinatorial optimization.
\newblock {\em Math. Programming}, 79(1-3, Ser. B):143--161, 1997.
\newblock Lectures on mathematical programming (ismp97) (Lausanne, 1997).

\bibitem{GroetschelLS93a}
M.~Gr{\"o}tschel, L.~Lov{\'a}sz, and A.~Schrijver.
\newblock {\em Geometric algorithms and combinatorial optimization}, volume~2
  of {\em Algorithms and Combinatorics}.
\newblock Springer-Verlag, Berlin, second edition, 1993.

\bibitem{HolstLS95a}
H.~van~der Holst, M.~Laurent, and A.~Schrijver.
\newblock On a minor-monotone graph invariant.
\newblock {\em J. Combin. Theory Ser. B}, 65(2):291--304, 1995.

\bibitem{HorvatKP11a}
B.~Horvat, J.~Kratochv{\'{\i}}l, and T.~Pisanski.
\newblock On the computational complexity of degenerate unit distance
  representations of graphs.
\newblock In {\em Combinatorial algorithms}, volume 6460 of {\em Lecture Notes
  in Comput. Sci.}, pages 274--285. Springer, Heidelberg, 2011.

\bibitem{KargerMS98a}
D.~Karger, R.~Motwani, and M.~Sudan.
\newblock Approximate graph coloring by semidefinite programming.
\newblock {\em J. ACM}, 45(2):246--265, 1998.

\bibitem{Knuth94a}
D.~E. Knuth.
\newblock The sandwich theorem.
\newblock {\em Electron. J. Combin.}, 1:Article 1, approx.\ 48 pp.\
  (electronic), 1994.

\bibitem{LaurentR05a}
M.~Laurent and F.~Rendl.
\newblock Semidefinite programming and integer programming.
\newblock In {\em Handbook on Discrete Optimization}, pages 393--514. Elsevier
  B. V., Amsterdam, 2005.

\bibitem{Lovasz79a}
L.~Lov{\'a}sz.
\newblock On the {S}hannon capacity of a graph.
\newblock {\em IEEE Trans. Inform. Theory}, 25(1):1--7, 1979.

\bibitem{LovaszSDP}
L.~Lov{\'a}sz.
\newblock Semidefinite programs and combinatorial optimization.
\newblock In {\em Recent advances in algorithms and combinatorics}, pages
  137--194. Springer, New York, 2003.

\bibitem{LovaszP86a}
L.~Lov{\'a}sz and M.~D. Plummer.
\newblock {\em Matching theory}, volume 121 of {\em North-Holland Mathematics
  Studies}.
\newblock North-Holland Publishing Co., Amsterdam, 1986.
\newblock Annals of Discrete Mathematics, 29.

\bibitem{McElieceRR78a}
R.~J. McEliece, E.~R. Rodemich, and H.~C. Rumsey, Jr.
\newblock The {L}ov\'asz bound and some generalizations.
\newblock {\em J. Combin. Inform. System Sci.}, 3(3):134--152, 1978.

\bibitem{Meurdesoif05a}
P.~Meurdesoif.
\newblock Strengthening the {L}ov\'asz {$\theta(\overline G)$} bound for graph
  coloring.
\newblock {\em Math. Program.}, 102(3, Ser. A):577--588, 2005.

\bibitem{Saxe80a}
J.~B. Saxe.
\newblock Two papers on graph embedding problems.
\newblock Technical Report CMU-CS-80-102, Department of Computer Science,
  Carnegie-Mellon University, 1980.

\bibitem{Schrijver79a}
A.~Schrijver.
\newblock A comparison of the {D}elsarte and {L}ov\'asz bounds.
\newblock {\em IEEE Trans. Inform. Theory}, 25(4):425--429, 1979.

\bibitem{Soifer09a}
A.~Soifer.
\newblock {\em The mathematical coloring book}.
\newblock Springer, New York, 2009.
\newblock Mathematics of coloring and the colorful life of its creators, With
  forewords by Branko Gr{\"u}nbaum, Peter D. Johnson, Jr. and Cecil Rousseau.

\bibitem{Szegedy94a}
M.~Szegedy.
\newblock A note on the theta number of {L}ov{\'a}sz and the generalized
  {D}elsarte bound.
\newblock In {\em Proceedings of the 35th {A}nnual {IEEE} {S}ymposium on
  {F}oundations of {C}omputer {S}cience}, 1994.

\end{thebibliography}
\end{document}